\documentclass[11pt]{article}
\usepackage[margin=1in]{geometry}
\usepackage{iftex}
\ifPDFTeX
  \usepackage[T1]{fontenc}
  \usepackage[utf8]{inputenc}
\fi
\usepackage{lmodern}
\usepackage{microtype}
\usepackage{mathtools,amssymb,amsthm,aliascnt}
\usepackage{booktabs,tabularx}
\usepackage{needspace}
\usepackage{xcolor}
\usepackage{xurl}
\usepackage{hyperref}
\usepackage[capitalize,nameinlink,noabbrev]{cleveref}
\definecolor{linkblue}{RGB}{25,55,100}
\hypersetup{colorlinks=true,linkcolor=linkblue,citecolor=linkblue,urlcolor=linkblue}
\numberwithin{equation}{section}
\newtheorem{theorem}{Theorem}[section]
\newaliascnt{corollary}{theorem}
\newtheorem{corollary}[corollary]{Corollary}
\aliascntresetthe{corollary}
\crefname{corollary}{Corollary}{Corollaries}
\Crefname{corollary}{Corollary}{Corollaries}
\newaliascnt{proposition}{theorem}
\newtheorem{proposition}[proposition]{Proposition}
\aliascntresetthe{proposition}
\crefname{proposition}{Proposition}{Propositions}
\Crefname{proposition}{Proposition}{Propositions}
\newaliascnt{lemma}{theorem}
\newtheorem{lemma}[lemma]{Lemma}
\aliascntresetthe{lemma}
\crefname{lemma}{Lemma}{Lemmas}
\Crefname{lemma}{Lemma}{Lemmas}
\newaliascnt{claim}{theorem}
\newtheorem{claim}[claim]{Claim}
\aliascntresetthe{claim}
\crefname{claim}{Claim}{Claims}
\Crefname{claim}{Claim}{Claims}
\title{Two-sided linear hashing and quadratic density bounds for smooth lattice coverings}
\hypersetup{pdftitle={Two-sided linear hashing and quadratic density bounds for smooth lattice coverings},pdfauthor={Ben Lund}}
\author{Ben Lund\\[0.4em]
  \small Xidian University\\
  \small Institute of Mathematics and Interdisciplinary Studies\\
  \small Xi'an, China}
\date{17 September 2026}
\begin{document}
\maketitle

\begin{abstract}
We study random linear projections of a finite-field subset for which
every fiber has cardinality close to its mean. We bound the mean fiber
size needed to ensure that all fibers satisfy a prescribed relative
discrepancy, with a prescribed failure probability. For
$S\subseteq\mathbb F_q^n$ projected to $\mathbb F_q^b$, one theorem gives
three regimes: at fixed discrepancy and failure probability, sufficient
mean fiber sizes are $O(q2^b)$ for arbitrary $q$, $O(q^2)$ when $q$ is at
least a suitable constant multiple of $b$, and $O_q(b)$ for fixed $q$.
The resulting entropy loss over fixed fields is
$h-b=\log_q h+O(1)$, where $h=\log_q|S|$ is the input entropy.
This matches the order of the binary obstruction of Alon,
Dietzfelbinger, Miltersen, Petrank, and Tardos~\cite{ADMPT}; we give a
quantitative refinement over every fixed field. Our proof
combines a quotient-and-average counting lemma with the local
balanced/unbalanced argument of Dhar and Dvir~\cite{DD} and Furstenberg
estimates of Dhar and Dvir~\cite{DD} and Kumar and Mon~\cite{KM}.

We apply these bounds in the reduction of Ordentlich, Regev, and
Weiss~\cite{ORW} to improve their $O(n^3)$ bound for smooth lattice
coverings to $O(n^2)$. For each fixed convex body
$K\subseteq\mathbb R^n$, a Haar--Siegel random lattice of covolume one
has the number of lattice points in every translate of $K$ within a
prescribed relative error of $\operatorname{vol}(K)$, with prescribed
high probability, once $\operatorname{vol}(K)\ge Cn^2$ and $n$ is
sufficiently large. The constant and dimension cutoff depend only on
the error and failure probability. Complements of higher-rank Kakeya
sets of Kopparty, Lev, Saraf, and Sudan~\cite{KLSS} show that no hashing
guarantee for arbitrary subsets can yield a smaller order in the same
reduction.
\end{abstract}

\section{Introduction}
\label{sec:introduction}

\subsection{Two-sided linear projections}
\label{sec:definitions}

Let \(q\) be a prime power, \(S\subseteq\mathbb F_q^n\) a nonempty set, and \(1\le b\le n\). Put
\[
B=q^b,\qquad
\mu=\frac{|S|}{q^b},\qquad h=\log_q|S|.
\]
For a linear map \(L:\mathbb F_q^n\to\mathbb F_q^b\), the mean of its fiber sizes on \(S\) is \(\mu\). Following Dhar and Dvir \cite{DD}, for \(\tau>0\) we call \(L\) \textbf{\(\tau\)-balanced} if
\begin{equation}
(1-\tau)\mu\le |S\cap L^{-1}(y)|\le(1+\tau)\mu
\qquad\text{for every }y\in\mathbb F_q^b,
\label{eq:I1}
\end{equation}
and \textbf{\(\tau\)-unbalanced} otherwise. If \(X\) is uniform on \(S\), this is equivalent to
\[
\|P_{LX}-U_b\|_\infty\le\tau q^{-b},
\]
where \(U_b\) is the uniform probability mass function on \(\mathbb F_q^b\). Thus we measure the error relative to the probability of an individual output. For \(0<\tau<1\), balance requires both nonempty fibers and counts close to their mean. For \(\tau\ge1\), only the upper inequality can fail.

Write \(\beta_S(b,\tau)\) for the proportion of unbalanced maps among the surjections. Given \(0<\delta<1\), we seek conditions implying
\begin{equation}
\beta_S(b,\tau)\le\delta
\label{eq:I2}
\end{equation}
for every \(S\) of the indicated size. We fix the source before sampling the random map, and a good map must balance all its fibers simultaneously. The relevant parameters are the field size \(q\), output dimension \(b\), kernel dimension \(n-b\), discrepancy \(\tau\), and failure probability \(\delta\). Since the input is uniform on \(S\), its entropy is \(h=\log_q|S|\) field symbols, and the entropy loss is
\[
h-b=\log_q\mu.
\]
A bound on \(\mu\) therefore gives a continuous entropy-loss threshold; selecting an integer output dimension can cost almost one additional symbol.

An affine \(r\)-flat is a translate of an \(r\)-dimensional linear subspace. We call that subspace its direction. Write \(\mathcal G_r(n)\) for the set of these directions. For \(\sigma>0\), we also call an affine \(r\)-flat \(R\) \(\sigma\)-balanced when
\[
\bigl||S\cap R|-|S|q^{r-n}\bigr|\le\sigma |S|q^{r-n}.
\]
This definition uses the ambient mean, even when we sample \(R\) inside a fixed larger flat. Following \cite[Definition 3.2]{DD}, we call a direction \(U\) \textbf{\(\sigma\)-shift-balanced} if every affine translate of \(U\) is \(\sigma\)-balanced. Thus a surjective linear map is \(\sigma\)-balanced exactly when its kernel is \(\sigma\)-shift-balanced. In the proof we pass from fibers of dimension \(n-b\) to flats of dimension \(n-b-a\). We call their dimension difference \(a\) the \textbf{descent dimension}.

\subsection{Kakeya, Furstenberg, and linear hashing}
\label{sec:background}

The finite-field Kakeya problem asks how small a subset of \(\mathbb F_q^d\) can be if it contains a complete affine line in every direction. Dvir \cite{Dvi} proved the lower bound \(c_dq^d\) using polynomials. Ellenberg, Oberlin, and Tao \cite{EOT} established the corresponding maximal estimate and introduced the higher-dimensional plane-maximal formulation in Section 4.12 and Conjecture 4.13. Requiring a complete affine \(r\)-flat in every direction gives the higher-rank Kakeya problem.

The Furstenberg problem relaxes completeness: for each direction, some affine \(r\)-flat need only meet the set in at least \(T\) points. We call such a flat \(T\)-rich, and call a direction \(T\)-rich if it admits such a translate. Ellenberg and Erman \cite{EE} developed the finite-field \(r\)-plane Furstenberg-set formulation and proved bounds of order \(T^{d/r}\), with dimension-dependent constants. Dhar, Dvir, and Lund \cite{DDL} gave elementary proofs and stronger bounds in other parameter ranges. In particular, for \(2\le r<d\), their Theorem 2 gives a lower bound arbitrarily close to \(Tq^{d-r}\) provided that \(T\) is sufficiently large.

Dhar and Dvir \cite{DD} connected these questions to two-sided linear hashing. In their argument, an unbalanced fiber yields many smaller unbalanced flats, and they apply a Furstenberg bound to this family. We follow their local balanced/unbalanced setup. One of our key contributions is to incorporate a quotient-and-average counting lemma that generalizes \cite[Lemma 21]{DDL}: for each direction of the smaller flats, we apply a partial-direction Furstenberg estimate in the quotient by that direction, and then average over these quotients. This counting step preserves the linear dependence on the fraction of rich directions. We combine it with the established geometric bounds \cite[Theorem 6.14]{DD} and \cite[Corollary 2.17]{KM} to obtain our projection theorems.

Linear hashing also has a history independent of this geometric approach. Alon, Dietzfelbinger, Miltersen, Petrank, and Tardos \cite{ADMPT} proved covering and maximum-load bounds, together with sources for which every linear map omits an output. Jaber, Kumar, and Zuckerman \cite{JKZ} obtained optimal-order binary maximum-load bounds and coarse two-sided balance in the dense regime. Kumar and Mon \cite{KM} established optimal-order expected maximum load over every fixed field. Pathegama and Barg \cite{PB} give average-norm guarantees for general sources, which do not directly imply pointwise relative balance. Doron, Leonov, Mosheiff, Navas, Resch, and Ribeiro \cite{DLMNRR} study relative discrepancy for translates of fixed nonnegative functions and for larger families under Fourier assumptions. We compare their indicator-set specialization below.

\subsection{Projection regimes and their predecessors}
\label{sec:corollaries}

We prove a master theorem, stated in \cref{sec:master}, whose bounds depend on the choice of descent dimension \(a\). Here we present four corollaries: three specializations for useful parameter ranges and a general mean bound that interpolates between field sizes. The corollaries use the notation of \cref{sec:definitions}, and their proofs appear in \cref{sec:corollary-proofs}.

\begin{corollary}[Improved entropy loss over very large fields]
\label{cor:very-large-fields}
Let \(q\) be a prime power, \(S\subseteq\mathbb F_q^n\) nonempty, \(1\le b\le n\), \(0<\tau\le1\), and \(0<\delta<1\). If
\begin{equation}
\mu\ge\frac{32q\,2^b}{\tau^2\delta},
\label{eq:H5}
\end{equation}
then \(\beta_S(b,\tau)<\delta\). The same failure bound holds for a uniform random \(b\times n\) matrix, not just surjections.
\end{corollary}

For fixed \(\tau,\delta\), the sufficient mean has order \(q2^b\), and the corresponding continuous loss is
\[
1+b\log_q2+\log_q\frac{32}{\tau^2\delta}.
\]
It approaches one symbol when \(b=o(\log q)\). The \(a=1\) descent bound improves the order of the \(q^2\) mean bound below when \(2^b=o(q)\). When \(|S|=q^j\), it permits \(b=j-2\) if \(j\ge3\) and \(q\ge32\,2^{j-2}/(\tau^2\delta)\). We do not claim that the dependence on output dimension is optimal.

\begin{corollary}[Quadratic mean at a linear field cutoff]
\label{cor:linear-field-cutoff}
Let \(q\) be a prime power, \(S\subseteq\mathbb F_q^n\) nonempty, \(1\le b\le n\), \(0<\tau\le1\), and \(0<\delta<1\). If
\begin{equation}
q\ge16(b+2),\qquad
\mu\ge\frac{192q^2}{\tau^2\delta},
\label{eq:H6a}
\end{equation}
then \(\beta_S(b,\tau)<\delta\). The same failure bound holds for a uniform random \(b\times n\) matrix, not just surjections.
\end{corollary}

\textbf{Comparison with Dhar and Dvir.} At fixed discrepancy and failure probability, \cref{cor:linear-field-cutoff} gives a continuous entropy-loss threshold \(2+o(1)\) already for field sizes linear in \(b\). By comparison, \cite[Theorem 2.1]{DD} gives three to four symbols of loss for fields linear in \(n\). Their Theorem 3.4 approaches two symbols, but achieving a threshold \(2+\xi\), for \(0<\xi\le1\), requires \(q^\xi\ge C_{\xi,\tau,\delta}n\). \cref{cor:very-large-fields} further reduces the continuous loss towards one symbol over sufficiently large fields.

The dependence on failure probability also improves: \cref{cor:linear-field-cutoff} requires a mean proportional to \(\delta^{-1}\), adding just \(\log_q(1/\delta)\) to the continuous entropy-loss threshold. In \cite[Theorems 2.1 and 3.4]{DD}, the respective hypotheses on \(q\) and \(q^\xi\) contain a factor \(\delta^{-2}\).

\begin{corollary}[Fixed fields and logarithmic entropy loss]
\label{cor:fixed-fields}
Let \(q\) be a prime power, let \(S\subseteq\mathbb F_q^n\) have entropy \(h=\log_q|S|\ge1\), and let \(0<\tau\le1\) and \(0<\delta<1\). Set
\begin{equation}
b=\left\lfloor
h-1-\left\lceil\log_q(16h)\right\rceil
-\log_q\frac{192}{\tau^2\delta}
\right\rfloor.
\label{eq:H7a}
\end{equation}
If \(b\ge1\), a uniform random linear map \(L:\mathbb F_q^n\to\mathbb F_q^b\) is \(\tau\)-balanced with probability greater than \(1-\delta\). The same conclusion holds for a uniform random surjection.
\end{corollary}

The next corollary gives a sufficient mean fiber size for a prescribed output dimension, with explicit dependence on the field size. It connects the large-field bound of \cref{cor:linear-field-cutoff} with the fixed-field regime of \cref{cor:fixed-fields} and covers intermediate field sizes.

\begin{corollary}[Mean bounds across field sizes]
\label{cor:mean-bounds}
Let \(q\) be a prime power, let \(S\subseteq\mathbb F_q^n\) be nonempty, and let \(1\le b\le n\), \(0<\tau\le1\), and \(0<\delta<1\). If there is an integer \(a\ge1\) such that
\begin{equation}
q^{a-1}\ge16(b+a),\qquad
\mu\ge\frac{192q^a}{\tau^2\delta},
\label{eq:H6}
\end{equation}
then \(\beta_S(b,\tau)<\delta\). The same failure bound holds for a uniform random \(b\times n\) matrix, not just surjections. In particular, for each fixed \(q\), there is a constant \(C_q\), independent of \(n,b,\tau,\delta\), such that the same conclusions hold whenever
\begin{equation}
\mu\ge\frac{C_q b}{\tau^2\delta}.
\label{eq:H7}
\end{equation}
\end{corollary}

The choice \(a=2\) recovers \cref{cor:linear-field-cutoff}; larger values cover smaller fields, with the fixed-field estimate giving the same logarithmic order of entropy loss as \cref{cor:fixed-fields}.

\textbf{Comparison with prior binary hashing bounds.} At fixed discrepancy and failure probability, \cite[Theorems 2.2 and 2.4]{DD} give losses \(4\log_2n+O(1)\) and \(4\log_2h+O(1)\), respectively. \cref{cor:fixed-fields} reduces the latter leading coefficient from four to one. The lower bound of Alon et al. \cite[Proposition 2.2]{ADMPT} shows that one is the optimal leading coefficient in the worst case for fixed \(0<\tau<1\) and \(0<\delta<1\); see \cref{sec:lower-bounds}. Dhar and Dvir obtain their binary bounds by reducing to a hashing bound over a large extension field, whereas we apply our master theorem directly over \(\mathbb F_2\).

Theorem 4 of \cite{JKZ} gives, for each \(0<\delta<1/2\), constants \(c_1,c_2\) such that
\[
\mu\ge c_1^{-1}\log_2B
\quad\Longrightarrow\quad
\Pr[\forall y,\ c_1\mu\le |S\cap L^{-1}(y)|\le c_2\mu]\ge1-\delta,
\]
where \(c_1=\Omega(\delta^{74})\), \(c_2=O(\delta^{-1/2})\). The fixed-field mean bound \labelcref{eq:H7} in \cref{cor:mean-bounds} subsumes this theorem as stated. Taking \(\tau=1/2\) gives factors \(1/2,3/2\) at mean \(4C_{\mathrm{bin}}\delta^{-1}\log_2B\), where \(C_{\mathrm{bin}}\ge1\) denotes our binary constant. Set \(c_1=\delta/(4C_{\mathrm{bin}})\), \(c_2=3/2\) to recover their formulation.

\begin{center}
\small
\begin{tabularx}{\textwidth}{@{}>{\raggedright\arraybackslash}X >{\centering\arraybackslash}p{0.24\textwidth} >{\raggedright\arraybackslash}p{0.26\textwidth}@{}}
\toprule
\textbf{Binary dense-regime guarantee} & \textbf{Sufficient mean} & \textbf{Lower and upper factors} \\
\midrule
\cite[Theorem 4]{JKZ}, displayed dependence & \(O(\delta^{-74}\log B)\) & \(\Omega(\delta^{74})\), \(O(\delta^{-1/2})\) \\
Bound \labelcref{eq:H7}, \(\tau=1/2\) & \(O(\delta^{-1}\log B)\) & \(1/2,\ 3/2\) \\
Bound \labelcref{eq:H7}, prescribed \(\tau\) & \(O(\tau^{-2}\delta^{-1}\log B)\) & \(1-\tau,\ 1+\tau\) \\
\bottomrule
\end{tabularx}
\end{center}

We therefore recover Theorem 4 of \cite{JKZ} with improved dependence on \(\delta\), and also allow any prescribed relative discrepancy \(\tau\). Their sparse-regime maximum-load theorems remain separate.

\textbf{Comparison with the code discrepancy bound of Doron et al.} For \(f=q^n\mathbf1_S/|S|\), \cite[Theorem 2.4]{DLMNRR} controls every count \(|S\cap(z+C)|\) relative to \(|C|\,|S|/q^n\), where \(C=\operatorname{im}G\) and \(G\) is a uniform \(n\times(n-b)\) matrix. Conditional on full rank, \(C\) has exactly the kernel distribution of a uniform surjection, so this is our all-fiber balance problem. Our results subsume this indicator-function specialization throughout its nonvacuous parameter range, reduce the required entropy loss, and improve the discrepancy and failure bounds. \cref{sec:quantitative-comparisons} gives the explicit bounds and verifies the subsumption for every field size.

\cref{cor:fixed-fields} substantially improves the accessible entropy-loss regime of \cite[Theorem 2.4]{DLMNRR}. Write \(\lambda=h-b>0\). Although their hypothesis is \(\lambda\ge240\log_qn\), their failure bound contains \(q^{2n-\lambda^2/1440}\), so it is nontrivial only when \(\lambda>\sqrt{2880n}\). For fixed \(q,\tau,\delta\), \cref{cor:fixed-fields} instead requires only \(\lambda=\log_qh+O_{q,\tau,\delta}(1)\).

We also improve the trade-off within their accessible loss regime. At the same \(\lambda\), we can reduce their relative discrepancy \(q^{-\lambda/12}\) to \(q^{-\lambda/6}\), while replacing their \(4q^{1-\lambda/3}\) failure term by \(O_q(hq^{-2\lambda/3})\) and eliminating the term \(q^{2n-\lambda^2/1440}\). In the generating-matrix model, our bound also includes an \(O(q^{-b})\) allowance for rank deficiency; this allowance disappears for a uniform code of dimension exactly \(n-b\). This comparison concerns indicator functions; we do not subsume their weighted-function or simultaneous-family theorems.

\subsection{The master theorem and the geometric inputs}
\label{sec:master}

The preceding bounds follow from one inequality, which can also be optimized between their displayed parameter choices.

\begin{theorem}[Master theorem]
\label{thm:master}
Let \(q\) be a prime power, let \(S\subseteq\mathbb F_q^n\) be nonempty, and let \(a,b\ge1\) be integers with \(b+a<n\). For \(\tau>0\), set
\[
E=\frac{|S|}{q^{b+a}},\qquad
\gamma=1-\frac{8(1+\tau)}{\tau^2E},
\]
and suppose that \(\gamma>0\). Then
\begin{equation}
\boxed{
\beta_S(b,\tau)\le
\frac{8}{\tau^2E}
\min\left\{
\left(\frac{1+q^{1-a}}{\gamma}\right)^{b+a},
\quad
\frac4\gamma
\exp\!\left(\frac{8(b+a)}{\gamma q^{a-1}}\right)
\right\}.
}
\label{eq:H1}
\end{equation}
We include the second expression in the minimum only when \(\gamma q^{a-1}\ge8\); otherwise we omit it.

For \(0<\tau\le1\), the same bound holds for the probability that a uniform random \(b\times n\) matrix over \(\mathbb F_q\) is \(\tau\)-unbalanced if we replace the leading constant \(8\) by \(9\).
\end{theorem}

The surjection bound allows every \(\tau>0\). If \(\tau\ge1\), negative discrepancy cannot cause a flat to be unbalanced, but the positive-discrepancy argument remains valid.

We call the first expression the \textbf{DD branch}; it uses \cite[Theorem 6.14]{DD}. We call the second expression the \textbf{KM branch}; it uses \cite[Corollary 2.17]{KM} and is unavailable for \(a=1\). In the latter corollary, if \(A\subseteq\mathbb F_q^d\) has a \(T\)-rich \(r\)-flat in a fraction \(\theta\) of directions and \(T\ge8q\), then
\[
|A|\ge\frac{\theta T}{4}q^{d-r}\exp(-8qd/T).
\]
Our contribution is the quotient-and-average extension of \cite[Lemma 21]{DDL} in \cref{clm:quotient-average}, which retains linear dependence on \(\theta\) while working in dimension \(d=b+a\). We combine it with the local descent of Dhar and Dvir (\cref{clm:global-unbalanced} and \cref{clm:unbalanced-parent}). The stronger geometric estimate of Kumar and Mon yields the quadratic mean bound, and the geometric reduction of Ordentlich, Regev, and Weiss converts this into the quadratic smooth-covering bound.

Even retaining only the DD branch, quotient averaging and parameter optimization give sufficient mean \(O_{\tau,\delta}(q^2b/\log b)\) when \(q\asymp b\). \cref{sec:quantitative-comparisons} records the short calculation. The KM branch removes this remaining \(b/\log b\) factor. At fixed \(q\), choosing \(a=\log_qb+O_q(1)\) gives the mean bound \labelcref{eq:H7}.

\subsection{Entropy-loss lower bounds over fixed fields}
\label{sec:lower-bounds}

Alon et al. \cite[Proposition 2.2]{ADMPT} construct a set of at least
\[
(\ln2)(b-3\log_2b)\,2^b
\]
points in \(\mathbb F_2^{b+\lfloor b/10\rfloor}\) whose image under every linear map to \(\mathbb F_2^b\) omits a point. Taking a subset of size \(\Theta(b2^b)\) shows that some sources of entropy \(h=b+\log_2b+O(1)\) require loss at least \(\log_2h-O(1)\) whenever \(\tau<1\). Their Theorem 7 gives covering at the matching order of mean, without near-uniform multiplicity.

The next theorem refines this lower bound and extends it to every fixed finite field, with exact cardinality, smaller ambient dimension, and an explicit dependence on the discrepancy. Its random-source method is related both to \cite[Proposition 2.2]{ADMPT} and to the positive-tail construction in \cite[Appendix A, Lemma A.1]{KM}.

\begin{theorem}[Exact-cardinality lower bound]
\label{thm:lower-bound}
Fix a prime power \(q\), a discrepancy \(0<\tau<1\), and \(0<\eta<1\). For every sufficiently large integer \(h\), set
\begin{equation}
\ell=\left\lfloor
\log_q\left(\frac{(1-\eta)h\ln q}{(1+\tau)\ln(1+\tau)-\tau}\right)
\right\rfloor,
\qquad b=h-\ell.
\label{eq:L10}
\end{equation}
There exist \(n=h+\log_q h+O_{q,\tau,\eta}(1)\) and \(S\subseteq\mathbb F_q^n\) with \(|S|=q^h\) such that every linear map \(\mathbb F_q^n\to\mathbb F_q^b\) is unbalanced. The same source defeats every map to more than \(b\) output symbols. In particular, taking \(\eta=1/2\), a universal guarantee, even if we choose the map optimally, requires entropy loss at least
\begin{equation}
\log_q h-O_{q,\tau}(1).
\label{eq:L11}
\end{equation}
\end{theorem}

Combining this with \cref{cor:fixed-fields} gives worst-case entropy loss
\[
\log_qh+O_{q,\tau,\delta}(1)
\]
for fixed \(q,\tau,\delta\). The upper bound holds at arbitrary source cardinalities; the lower theorem holds at every sufficiently large cardinality \(q^h\) with integer \(h\), and defeats every map. This establishes the leading-order obstruction to any universal smaller loss. It does not establish optimal dependence on a varying failure probability or discrepancy.

\subsection{Smooth lattice coverings}
\label{sec:smooth-coverings}

For a full-rank lattice \(\Lambda\subseteq\mathbb R^n\), a compact convex body \(K\) with nonempty interior, and \(x\in\mathbb R^n\), write
\begin{equation}
N(\Lambda,K,x)=|\Lambda\cap(K+x)|,\qquad
\eta(K,\Lambda)=\sup_x
\left|\frac{N(\Lambda,K,x)}{\operatorname{vol}(K)/\det\Lambda}-1\right|.
\label{eq:C1}
\end{equation}
The mean multiplicity, or density, is \(\operatorname{vol}(K)/\det\Lambda\). We call the lattice covering \(\varepsilon\)-smooth if \(\eta(K,\Lambda)<\varepsilon\). Since \(\Lambda=-\Lambda\), this is also the relative variation in the number of translates \(\lambda+K\), \(\lambda\in\Lambda\), covering a point.

For smooth lattice coverings, \cite[Theorem 1.2]{ORW} is the direct predecessor. Their method discretizes the body, applies two-sided linear hashing, and uses invariance of random lattice neighbors. We retain this geometric reduction and improve its finite-field input. Let \(\mu_n\) denote Haar--Siegel probability, the invariant probability measure on full-rank lattices of determinant one (unimodular lattices); \(\varepsilon\) and \(\Delta\) denote the continuous discrepancy and failure probability.

The reduction produces finite-field sources with \(p\asymp_\varepsilon n\) and mean fiber size comparable to the target covering density. A quadratic mean bound gives quadratic covering density. We explain the reduction and our modifications in \cref{sec:covering-proof}.

The main ingredient that makes the \(O(n^2)\) bound possible is the Furstenberg estimate of Kumar and Mon \cite[Corollary 2.17]{KM}. With discrepancy and failure probability fixed, it gives the \(O(p^2)\) mean bound in \cref{cor:linear-field-cutoff}. If we instead use the earlier Furstenberg estimate of Dhar and Dvir, our optimized quotient-and-average argument gives a sufficient mean of \(O(p^2b/\log b)\), as shown in \cref{sec:quantitative-comparisons}. Since \(p\asymp n\) and \(b\asymp n\) at polynomial target densities, this yields a smooth-covering bound of \(O(n^3/\log n)\).

\begin{theorem}[quadratic smooth-covering bound]
\label{thm:smooth-covering}
There is an absolute constant \(C\) such that, for every \(0<\varepsilon,\Delta<1\), there is \(n_0=n_0(\varepsilon,\Delta)\) with the following property. For every \(n\ge n_0\) and every compact convex body \(K\subseteq\mathbb R^n\) with nonempty interior and
\begin{equation}
\operatorname{vol}(K)\ge
C\varepsilon^{-4}\Delta^{-1}n^2,
\label{eq:C2}
\end{equation}
one has
\begin{equation}
\Pr_{\Lambda\sim\mu_n}\bigl(\eta(K,\Lambda)\ge\varepsilon\bigr)<\Delta.
\label{eq:C3}
\end{equation}
We may take \(C=2^{52}\) and any integer \(n_0\ge5\) satisfying
\[
2^{19}n_0/\varepsilon\le2^{n_0},
\qquad 3\cdot2^{-n_0}<\Delta/2.
\]
We assume neither symmetry nor smoothness of \(K\). The dimension cutoff is uniform in \(K\).
\end{theorem}

\textbf{Comparison with Ordentlich, Regev, and Weiss.} Their Theorem 1.2 gives the same probability conclusion for
\begin{equation}
\operatorname{vol}(K)\ge2^{66}(\varepsilon\Delta)^{-6.5}n^3,
\qquad n>25.
\label{eq:C4}
\end{equation}
\cref{thm:smooth-covering} reduces the dimension dependence from cubic to quadratic, and the displayed tolerance and failure costs to \(\varepsilon^{-4}\Delta^{-1}\). The new result has the qualification \(n\ge n_0(\varepsilon,\Delta)\); we do not claim that it improves the numerical threshold in every dimension. In \cref{cor:existence-density}, we take \(\Delta=1/2\) and rescale to obtain smooth lattice coverings at density \(O(\varepsilon^{-4}n^2)\).

\subsection{\texorpdfstring{The barrier for unrestricted hashing and Remark 3.7 of \cite{ORW}}{The barrier for unrestricted hashing and Remark 3.7 of Ordentlich, Regev, and Weiss}}
\label{sec:barrier}

At polynomial target volumes, the reduction of Ordentlich, Regev, and Weiss from smooth lattice coverings to finite-field hashing uses kernel dimension of order \(n/\ln p\). In this parameter regime, quadratic order is the limit of replacing the finite-field input by a stronger guarantee for arbitrary subsets. The construction providing the obstruction is \cite[Theorem 4]{KLSS}, specialized to rank four. For primes \(p\ge5\), it gives a set \(A\subseteq\mathbb F_p^d\) containing a translate of every four-dimensional subspace and satisfying
\[
|A|\le p^d
\left(1-\frac{p-3}{2p^4}\right)^{\lfloor d/5\rfloor}.
\]
An appropriate subset of its complement has an empty fiber under every map to \(\mathbb F_p^{d-4}\). Adding zero coordinates gives the growing kernel dimensions below.

\begin{corollary}[Quadratic obstruction in the covering regime]
\label{cor:quadratic-obstruction}
Fix \(C_0,c>0\). Along primes \(p\to\infty\), set
\begin{equation}
n=\lfloor p/C_0\rfloor,\qquad
r=\left\lceil\frac{cn}{\ln p}\right\rceil,\qquad
b=n-r.
\label{eq:L18}
\end{equation}
For all sufficiently large \(p\), there is a nonempty \(S\subseteq\mathbb F_p^n\) such that every linear map to \(\mathbb F_p^b\) has an empty fiber and
\begin{equation}
\mu=\frac{|S|}{p^b}
=(1+o(1))\frac{np}{10}
=(1+o(1))\frac{p^2}{10C_0},
\qquad
\frac{|S|}{p^n}=\exp(-cn+O(\ln p)).
\label{eq:L19}
\end{equation}
For every positive integer \(m\le |S|\), we can choose an \(m\)-element subset of \(S\) with the same empty-fiber property. Consequently, for fixed \(0<\tau,\delta<1\), no sufficient mean threshold \(o(n^2)\) can guarantee a balanced fraction \(1-\delta\) for every subset in this parameter range.
\end{corollary}

Remark 3.7 of \cite{ORW} considers an unrestricted hypothesis requiring balance whenever \(p>p_*(n)\) and
\begin{equation}
r>m_*+n-\log_p|S|,
\qquad\text{equivalently}\qquad \mu>p^{m_*}.
\label{eq:L21}
\end{equation}
If \(p_*(n)\le C_*n\), choose \(C_0>C_*\) in \cref{cor:quadratic-obstruction}. Its source satisfies \(\mu>p^{m_*}\) for every fixed \(m_*<2\), while the balanced fraction is zero, even at \(\tau=\delta=1/2\). Thus the proposed finite-field hypothesis is false with a linear field cutoff and exponent below two. The growing kernel dimension and sparse density in \labelcref{eq:L19} show that imposing those two features alone does not remove the counterexample.

This does not invalidate the conditional implication of Ordentlich, Regev, and Weiss. \cref{cor:quadratic-obstruction} excludes improving the quadratic order solely by an unrestricted all-set hashing bound in their reduction, including a saving by any factor tending to infinity. It is not a quadratic lower bound for smooth lattice coverings themselves. Additional properties of discretized convex bodies or a different reduction can lie outside this obstruction.

The same construction explains the integer-output rounding in \cref{cor:linear-field-cutoff}. With rank four, \(n=\lfloor p/C_0\rfloor\), and \(C_0\ge32\), its bad target \(b=n-4\) has entropy
\[
h=b+2-\log_p(10C_0)+o(1).
\]
Every successful integer output is at most \(b-1\), requiring loss at least \(3-o(1)\). \cref{cor:linear-field-cutoff} does apply at \(b-1\) for sufficiently large \(p\), where the mean is of order \(p^3\). Thus \cref{cor:linear-field-cutoff} attains the largest possible integer output on these examples: its loss is below \(3+o(1)\), while its continuous threshold is \(2+o(1)\).

\subsection{Organization}
\label{sec:organization}

\cref{sec:projection-proofs} proves the master theorem, derives the four projection corollaries, and treats uniform random matrices. \cref{sec:lower-proof} proves the exact-cardinality lower bound. \cref{sec:kakeya-obstruction} deduces the Kakeya obstructions from \cite{KLSS}. \cref{sec:covering-proof} proves the smooth-covering theorem using the discretization and invariant-measure inputs of \cite{ORW} and \cite{ORW22}. All logarithms without an indicated base are natural.

\section{Proofs of the projection theorems}
\label{sec:projection-proofs}

We use the notation and balanced/unbalanced definitions of \cref{sec:definitions}. \cref{sec:master-proof} proves the master theorem for surjections, \cref{sec:corollary-proofs} derives the four projection corollaries, and \cref{sec:rank-transfer} proves the extension to unrestricted random matrices. \cref{sec:quantitative-comparisons} verifies the quantitative comparisons from the introduction.

\subsection{Proof of the master theorem for surjections}
\label{sec:master-proof}

We use two geometric estimates. For \(X\subseteq\mathbb F_q^d\), we call an affine \(a\)-flat \(R\) \textbf{\(\gamma\)-dense} if \(X\) has density at least \(\gamma\) on \(R\), that is, \(|X\cap R|\ge\gamma|R|=\gamma q^a\).

The first input is \cite[Theorem 6.14]{DD}: if a \(\theta\)-fraction of the \(a\)-directions have a \(\gamma\)-dense translate, then
\begin{equation*}
\frac{|X|}{q^d}\ge
\theta\left(\frac{\gamma}{1+q^{1-a}}\right)^d.
\tag{DD}\label{eq:H-DD}
\end{equation*}
Here \(1\le a\le d\) and \(0\le\theta,\gamma\le1\). The second input is \cite[Corollary 2.17]{KM}: under the same density hypothesis, if \(\gamma q^{a-1}\ge8\), then
\begin{equation*}
\frac{|X|}{q^d}\ge
\frac{\theta\gamma}{4}
\exp\!\left(-\frac{8d}{\gamma q^{a-1}}\right).
\tag{KM}\label{eq:H-KM}
\end{equation*}
This is Corollary 2.17 of \cite{KM} with their parameters \(n=d\), \(r=a\), \(\delta=\theta\), and \(T=\gamma q^a\), after dividing by \(q^d\). Their hypothesis \(T\ge8q\) is exactly \(\gamma q^{a-1}\ge8\). If the source theorem requires an integer threshold, we use \(T=\lceil\gamma q^a\rceil\); the conclusion above follows because \(T\exp(-8qd/T)\) is increasing for \(T>0\). The case \(\theta=0\) is trivial.

\Needspace{8\baselineskip}
Our first step is to convert these estimates for point sets into a counting bound for families of affine flats. Quotienting by a direction turns its affine translates into points, so we can apply a Furstenberg estimate in each quotient and then average. The following claim generalizes the quotient counting step in \cite[Lemma 21]{DDL}: we allow both a fraction of pairs of nested directions and a fraction of the parallel smaller flats in each witnessing translate.

\begin{claim}[Quotient-and-average counting]
\label{clm:quotient-average}
Let \(1\le\ell<k<n\), \(0\le\alpha\le1\), and \(0<\gamma\le1\), and put \(a=k-\ell\) and \(d=n-\ell\). Let \(\mathcal B\) be a family of affine \(\ell\)-flats in \(\mathbb F_q^n\), of density \(p_{\mathcal B}\) among all such flats. Suppose that an \(\alpha\)-fraction of the containment flags
\[
(W,U)\in\mathcal G_\ell(n)\times\mathcal G_k(n),\qquad W\subset U,
\]
admit an affine translate of \(U\) containing at least \(\gamma q^a\) members of \(\mathcal B\) parallel to \(W\). Then
\begin{equation}
p_{\mathcal B}\ge\alpha
\max\left\{
\left(\frac{\gamma}{1+q^{1-a}}\right)^d,
\quad
\frac{\gamma}{4}
\exp\!\left(-\frac{8d}{\gamma q^{a-1}}\right)
\right\},
\label{eq:H2}
\end{equation}
where we omit the second expression unless \(\gamma q^{a-1}\ge8\).
\end{claim}

\begin{proof}
For each \(W\in\mathcal G_\ell(n)\), its affine translates are the points of the quotient \(\mathbb F_q^n/W\), a vector space of dimension \(d\). Let
\[
X_W=\{x+W:x+W\in\mathcal B\}.
\]
Let \(\theta_W\) be the fraction of \(k\)-directions \(U\supset W\) for which the hypothesis holds. The correspondence
\[
U\longmapsto U/W
\]
is a bijection from these \(k\)-directions to the \(a\)-directions in the quotient. Indeed, its inverse takes a quotient subspace to its full inverse image in \(\mathbb F_q^n\).

If \(z+U\) witnesses the hypothesis for \((W,U)\), its image in the quotient is the affine \(a\)-flat \((z+W)+U/W\). Its \(q^a\) points are precisely the translates of \(W\) inside \(z+U\); at least \(\gamma q^a\) belong to \(X_W\). Thus \(X_W\) has a \(\gamma\)-dense translate in a \(\theta_W\)-fraction of its \(a\)-directions. Each of \labelcref{eq:H-DD} and, when admissible, \labelcref{eq:H-KM} bounds \(|X_W|/q^d\) below by \(\theta_W\) times the corresponding expression in \labelcref{eq:H2}.

The uniform containment-flag measure has uniform marginals on \(\mathcal G_\ell(n)\) and \(\mathcal G_k(n)\). This follows because each subspace of either dimension lies in the same number of flags. Hence
\[
\mathbb E_W\theta_W\ge\alpha.
\]
Every \(W\) has exactly \(q^d\) translates, and each affine flat has a unique direction. Consequently
\[
p_{\mathcal B}=\mathbb E_W\frac{|X_W|}{q^d}.
\]
Averaging the two geometric inequalities proves the claim. \end{proof}

\Needspace{8\baselineskip}
To apply this counting bound, we first show that unbalanced flats are rare: a random affine flat has intersection size concentrated near its global mean when that mean is large.

\begin{claim}[Global density of unbalanced flats]
\label{clm:global-unbalanced}
Let \(1\le\ell<n\), let \(E=|S|q^{\ell-n}\), and let \(\sigma>0\). The density of \(\sigma\)-unbalanced affine \(\ell\)-flats is at most
\begin{equation}
\frac1{\sigma^2E}.
\label{eq:H3}
\end{equation}
\end{claim}

\begin{proof}
For a uniform affine \(j\)-flat in an affine \(m\)-space and distinct points \(x,y\),
\[
\Pr[x,y\text{ both belong to the flat}]
=q^{j-m}\frac{q^j-1}{q^m-1}
\le q^{2(j-m)}.
\]
Thus the intersection count of a fixed set with such a flat has variance at most its mean. In the present case that mean is \(E\), and Chebyshev's inequality gives \labelcref{eq:H3}. \end{proof}

\Needspace{8\baselineskip}
Inside an unbalanced parent flat, the intersection size of a random subflat concentrates around a mean displaced from the global mean, so balance at a stricter tolerance is unlikely when the expected counts are large.

\begin{claim}[Unbalanced parents have few balanced subflats]
\label{clm:unbalanced-parent}
Let \(1\le\ell<k<n\), put \(a=k-\ell\) and \(E=|S|q^{\ell-n}\), and fix \(0<\sigma<\tau\). If \(T\) is a \(\tau\)-unbalanced affine \(k\)-flat, then a uniform affine \(\ell\)-flat \(R\subset T\) is \(\sigma\)-balanced with probability at most
\begin{equation}
\frac{1+\tau}{(\tau-\sigma)^2E}.
\label{eq:H4}
\end{equation}
\end{claim}

\begin{proof}
Put \(X=|S\cap R|\) and \(u=|S\cap T|/q^a\). The variance estimate in \cref{clm:global-unbalanced}, now inside \(T\), gives \(\operatorname{Var}X\le u\). If \(u>(1+\tau)E\), balance requires \(X\le(1+\sigma)E\), so
\[
\Pr[R\text{ balanced}]
\le\frac{u}{(u-(1+\sigma)E)^2}
\le\frac{1+\tau}{(\tau-\sigma)^2E}.
\]
The last step uses that \(u/(u-c)^2\) decreases for \(u>c\ge0\).

The other possibility is \(u<(1-\tau)E\), which can occur only when \(\tau<1\). Then \(\sigma<1\), and balance requires \(X\ge(1-\sigma)E\). The same estimate gives
\[
\Pr[R\text{ balanced}]
\le\frac{u}{((1-\sigma)E-u)^2}
\le\frac{1-\tau}{(\tau-\sigma)^2E},
\]
using that \(u/(c-u)^2\) increases for \(0\le u<c\). This is no larger than \labelcref{eq:H4}. \end{proof}

\begin{proof}[Completion of the proof of \cref{thm:master} for surjections]
Set
\[
k=n-b,\qquad \ell=k-a,\qquad \sigma=\tau/2,
\]
and let \(\mathcal B\) be the family of \(\sigma\)-unbalanced affine \(\ell\)-flats. \cref{clm:global-unbalanced} gives \(p_{\mathcal B}\le4/(\tau^2E)\).

A uniform surjection has a uniform \(k\)-dimensional kernel: for each kernel there are equally many isomorphisms from its quotient to \(\mathbb F_q^b\). A kernel fails to be \(\tau\)-shift-balanced exactly when at least one of its affine translates is \(\tau\)-unbalanced. For each such direction \(U\), select one \(\tau\)-unbalanced translate \(T_U\).

For \(W\in\mathcal G_\ell(n)\) with \(W\subset U\), let \(f_U(W)\) be the fraction of the \(q^a\) translates of \(W\) inside \(T_U\) that are \(\sigma\)-balanced. Uniformly sampling \(W\subset U\) and then one of these translates samples a uniform affine \(\ell\)-subflat of \(T_U\). \cref{clm:unbalanced-parent} therefore gives
\[
\mathbb E_{W\subset U}f_U(W)\le
\varepsilon_0:=\frac{4(1+\tau)}{\tau^2E}.
\]
Markov's inequality shows that at least half these \(W\) have \(f_U(W)\le2\varepsilon_0\). For each such flag, at least
\[
(1-2\varepsilon_0)q^a=\gamma q^a
\]
parallel subflats of \(T_U\) belong to \(\mathcal B\).

Thus \cref{clm:quotient-average} applies with \(\alpha\ge\beta_S(b,\tau)/2\) and \(d=b+a\). Combining its lower bound with \cref{clm:global-unbalanced} and inverting the positive geometric factors proves \labelcref{eq:H1}. \end{proof}

\subsection{Proofs of the projection corollaries}
\label{sec:corollary-proofs}

Throughout these proofs, \(L\) is uniform among either all linear maps or the surjections with the indicated domain and codomain. We use \cref{thm:master} with leading constant \(9\); \cref{sec:rank-transfer} proves its unrestricted-matrix extension.

\begin{proof}[Proof of \cref{cor:very-large-fields}]
Use \(a=1\), so \(E=\mu/q\). Since \(E>1\) and \(|S|\le q^n\), the source-size hypothesis implies \(n>b+1\). Put \(d=b+1\). Because \(d\le2^b\),
\[
d(1-\gamma)
=\frac{8(1+\tau)d}{\tau^2E}
\le\frac{\delta}{2}.
\]
In particular \(\gamma>0\). Bernoulli's inequality with exponent \(d/2\ge1\) gives \(\gamma^{d/2}\ge1-\delta/4\). The first branch of \cref{thm:master}, with leading constant \(9\), yields
\[
\Pr[L\text{ is }\tau\text{-unbalanced}]
\le\frac{9\,2^{b+1}}{\tau^2E\gamma^{b+1}}
\le\frac{9\delta}{16(1-\delta/4)^2}<\delta,
\]
where the last inequality uses \(\delta<1\). \end{proof}

\begin{proof}[Proof of \cref{cor:mean-bounds}]
Set \(E=|S|/q^{b+a}\). The size hypothesis gives \(E\ge192/(\tau^2\delta)>1\), and hence \(b+a<n\). Also \(1-\gamma\le(1+\tau)\delta/24<1/12\), so in particular \(\gamma\ge3/4\). The field condition implies \(\gamma q^{a-1}\ge8\) and
\[
\frac{8(b+a)}{\gamma q^{a-1}}\le1.
\]
The second branch of \cref{thm:master}, with leading constant \(9\), therefore gives
\begin{equation}
\Pr[L\text{ is }\tau\text{-unbalanced}]
\le\frac{48e}{\tau^2E}
\le\frac e4\,\delta<\delta.
\label{eq:H6b}
\end{equation}

For the fixed-field conclusion \labelcref{eq:H7}, take the least positive integer \(a\) with \(q^{a-1}\ge16(b+a)\). It is \(\log_qb+O_q(1)\): the lower bound follows from \(q^{a-1}\ge16b\), and \(a=\lceil\log_qb\rceil+c_q\) is admissible for every \(b\ge1\) when \(c_q\) is sufficiently large. Thus \(q^a=O_q(b)\), and \labelcref{eq:H6} proves \labelcref{eq:H7}. \end{proof}

\begin{proof}[Proof of \cref{cor:linear-field-cutoff}]
Apply \cref{cor:mean-bounds} with \(a=2\). \end{proof}

\begin{proof}[Proof of \cref{cor:fixed-fields}]
Set \(a=1+\lceil\log_q(16h)\rceil\) and \(E_0=192/(\tau^2\delta)\). The choice \labelcref{eq:H7a} gives \(b=\lfloor h-a-\log_qE_0\rfloor\), so \(b+a\le h-\log_qE_0<h\le n\) and
\[
q^{a-1}\ge16h\ge16(b+a),
\qquad
\frac{|S|}{q^{b+a}}\ge E_0.
\]
Apply \cref{cor:mean-bounds}, which proves the corollary. \end{proof}

The choice of descent dimension satisfies \(a<2+\log_q(16h)\), and the definition of \(b\) gives \(b>h-a-\log_qE_0-1\). Thus
\begin{equation}
b>h-\log_qh-\log_q\frac1{\tau^2\delta}-3-\log_q3072.
\label{eq:H7b}
\end{equation}
In particular, over \(\mathbb F_2\),
\begin{equation}
b>h-\log_2h-2\log_2(1/\tau)-\log_2(1/\delta)-15,
\label{eq:H7c}
\end{equation}
since \(3+\log_2 3072<15\).

\subsection{Extension to unrestricted random matrices}
\label{sec:rank-transfer}

\begin{proposition}[Rank transfer]
\label{prop:rank-transfer}
For an arbitrary matrix, define balance by the same all-fiber inequality as for a surjection. Let \(A\) be uniform among all \(b\times n\) matrices over \(\mathbb F_q\), with \(1\le b\le n\), and put
\[
r_{n,b,q}=1-\prod_{i=0}^{b-1}(1-q^{i-n}).
\]
Conditional on full row rank, \(A\) is a uniform surjection. Therefore
\begin{equation}
\Pr[A\text{ is }\tau\text{-unbalanced}]
\le\beta_S(b,\tau)+(1-\beta_S(b,\tau))r_{n,b,q}
\le\beta_S(b,\tau)+r_{n,b,q},
\label{eq:H8}
\end{equation}
and
\begin{equation}
r_{n,b,q}\le
\sum_{i=0}^{b-1}q^{i-n}
=\frac{q^b-1}{(q-1)q^n}
<\frac{q^{b-n}}{q-1}.
\label{eq:H8a}
\end{equation}
For \(0<\tau<1\), every rank-deficient matrix is unbalanced, so the first inequality in \labelcref{eq:H8} is an equality.
\end{proposition}

\begin{proof}
Exposing the rows successively gives the displayed formula for full row rank. The bound \labelcref{eq:H8a} follows from \(1-\prod_i(1-x_i)\le\sum_i x_i\). If the rank is deficient, there is an empty output fiber; for \(0<\tau<1\), its discrepancy is greater than \(\tau\mu\). \end{proof}

\begin{proof}[Proof of the unrestricted-matrix extension in \cref{thm:master}]
Assume \(0<\tau\le1\), and let \(R\) denote the right-hand side of \labelcref{eq:H1}. Each expression in its minimum is at least one, so \(R\ge8/(\tau^2E)\). Since \(|S|\le q^n\), \cref{prop:rank-transfer} gives
\[
r_{n,b,q}<\frac1{(q-1)q^aE}
\le\frac1{\tau^2E}\le\frac R8.
\]
Applying \cref{prop:rank-transfer} once more, we obtain
\[
\Pr[A\text{ is }\tau\text{-unbalanced}]
\le R+r_{n,b,q}\le\frac98R,
\]
which replaces the leading constant \(8\) by \(9\). \end{proof}

\subsection{Quantitative comparisons}
\label{sec:quantitative-comparisons}

\textbf{Comparison with Doron et al.} For \(h\ge1\), \(\lambda=h-b>0\), and \(0<\tau\le1\), \cref{cor:fixed-fields} implies
\[
\beta_S(b,\tau)\le
\min\left\{1,\frac{3072q^2h}{\tau^2}q^{-\lambda}\right\}.
\]
To see this, set \(a=1+\lceil\log_q(16h)\rceil\), so \(192q^a\le3072q^2h\). If \(\delta=3072q^2h\,\tau^{-2}q^{-\lambda}<1\), the output dimension prescribed by \cref{cor:fixed-fields} is at least \(b\). Composing with a fixed surjection to \(\mathbb F_q^b\) preserves balance and the uniform distribution on surjections. If \(\delta\ge1\), the bound is trivial.

Now let \(G\) be a uniform \(n\times(n-b)\) matrix and \(C=\operatorname{im}G\), as in \cite[Section 2.1]{DLMNRR}. Conditional on full rank, the preceding bound applies to the relative errors in all counts \(|S\cap(z+C)|\), normalized by \(|C|\,|S|/q^n\). The rank-deficiency probability is less than \(q^{-b}/(q-1)\). Thus, at discrepancy \(q^{-\lambda/6}\), the failure probability is at most
\[
3072q^2h\,q^{-2\lambda/3}+\frac{q^{-b}}{q-1}.
\]
The corresponding bounds in \cite[Theorem 2.4]{DLMNRR} are discrepancy \(q^{-\lambda/12}\) and failure at most
\[
4q^{1-\lambda/3}
+\frac{q^{1-b}}{(q-1)^2}
+q^{2n-\lambda^2/1440},
\qquad \lambda\ge240\log_qn.
\]
Whenever this failure bound is less than one, \(\lambda>\sqrt{2880n}>6\) and \(n\ge2\), since \(\lambda\le h\le n\). Using \(h\le n\) and the displayed hypothesis on \(\lambda\), we obtain
\[
\frac{3072q^2h\,q^{-2\lambda/3}}{4q^{1-\lambda/3}}
=768qh\,q^{-\lambda/3}
\le768n\,q^{-\lambda/6}
\le768n^{-39}<1.
\]
Also \(q^{-b}/(q-1)<q^{1-b}/(q-1)^2\). Hence our bound gives both a smaller discrepancy and a smaller failure probability throughout the nonvacuous range of their indicator specialization. The case \(b=0\), outside our projection convention, follows directly: full rank gives \(C=\mathbb F_q^n\) and zero discrepancy, leaving only the same rank-deficiency bound.

\textbf{Optimizing the DD branch.}

For \(a=2\), fixed \(\tau,\delta\), and \(q\asymp b\), we verify the estimate stated in \cref{sec:master} using only the DD branch. Write \(d=b+2\), \(A=8(1+\tau)/\tau^2\), and \(\beta=\beta_S(b,\tau)\). The bound is
\[
\beta\le\frac8{\tau^2 E}
(1-A/E)^{-d}(1+1/q)^d.
\]
For \(E\ge2A\), the last two factors are at most
\[
\exp(2Ad/E+d/q).
\]
Choosing \(E=C_\tau d/\log d\), with \(C_\tau>4A\), bounds the failure probability by a constant times
\[
(\log d)d^{-1+2A/C_\tau},
\]
which tends to zero. Thus the DD branch suffices at mean \(O_{\tau,\delta}(q^2b/\log b)\).

\section{Proof of the fixed-field lower bound}
\label{sec:lower-proof}

We prove \cref{thm:lower-bound} by choosing a source of prescribed cardinality and taking a union bound over all kernels. This is a refinement of the random-source methods of \cite[Proposition 2.2]{ADMPT} and \cite[Appendix A]{KM}. The proof keeps track of the positive-discrepancy rate and the ambient dimension.

\begin{lemma}[exact-cardinality random sources]
\label{lem:random-sources}
Let \(1\le b<n\), set
\[
B=q^b,\qquad F=q^{n-b},\qquad N=BF,
\]
and let \(1\le m<N\) be an integer. Put \(\mu=m/B\), fix \(\tau>0\), and define
\begin{equation}
k=\lfloor(1+\tau)\mu\rfloor+1,
\qquad
p_+=\sum_{j\ge k}
\frac{\binom Fj\binom{N-F}{m-j}}{\binom Nm}.
\label{eq:L1}
\end{equation}
There is an \(m\)-element set for which the balanced fraction among surjections is at most \(e^{-Bp_+}\). Moreover, if
\begin{equation}
{n\brack b}_q e^{-Bp_+}<1,
\label{eq:L2}
\end{equation}
then some \(m\)-element set has a fiber larger than \((1+\tau)\mu\) under every surjection to \(\mathbb F_q^b\).
\end{lemma}

\begin{proof}
Choose \(S\) uniformly among all \(m\)-element subsets of \(\mathbb F_q^n\). For one fixed surjection let \(X_1,\ldots,X_B\) be the fiber counts. Each has the hypergeometric distribution in \labelcref{eq:L1}. We claim that
\begin{equation}
\Pr(X_i\le k-1\text{ for all }i)\le(1-p_+)^B.
\label{eq:L3}
\end{equation}
For an index set \(J\) and \(i\in J\), conditional on \(X_i=x\), the source outside the \(i\)-th fiber is a uniform \((m-x)\)-subset of its complement. The event \(X_j\le k-1\) for every \(j\in J\setminus\{i\}\) is decreasing under addition of source points. Coupling subsets of different sizes by initial segments of a uniform permutation shows that its conditional probability \(g(x)\) is nondecreasing in \(x\). Since \(f(x)=\mathbf1_{\{x\le k-1\}}\) is nonincreasing,
\[
\operatorname{Cov}(f(X_i),g(X_i))\le0.
\]
For example, this follows by expanding the expectation of \((f(X_i)-f(X_i'))(g(X_i)-g(X_i'))\le0\), with \(X_i'\) an independent copy. Consequently the probability of all the inequalities indexed by \(J\) is at most \((1-p_+)\) times the probability of those indexed by \(J\setminus\{i\}\). Induction proves \labelcref{eq:L3}.

Balance implies the event in \labelcref{eq:L3}. Averaging over the kernels proves the first assertion. Balance depends only on the kernel, and there are \({n\brack b}_q\) kernels. The union bound and \labelcref{eq:L2} prove the second assertion. \end{proof}

For \(\tau>0\), write
\begin{equation}
I_+(\tau)=(1+\tau)\ln(1+\tau)-\tau.
\label{eq:L4}
\end{equation}

\begin{lemma}[a sufficient reverse-tail estimate]
\label{lem:reverse-tail}
Fix \(0<\tau\le1\), with \(I_+(\tau)>0\) as in \labelcref{eq:L4}.
There is an absolute constant \(c>0\) such that, whenever
\begin{equation}
\mu\ge1,\qquad F\ge32\mu^2,\qquad B\ge4F,
\label{eq:L5}
\end{equation}
the probability in \labelcref{eq:L1} satisfies
\begin{equation}
p_+\ge c\mu^{-1/2}\exp(-\mu I_+(\tau)).
\label{eq:L6}
\end{equation}
\end{lemma}

\begin{proof}
We need only bound the single point probability at \(k\). Let \(X\) have the hypergeometric distribution in \labelcref{eq:L1}, let \(Y\sim\operatorname{Bin}(F,\mu/F)\), and let \(Z\sim\operatorname{Pois}(\mu)\). The hypotheses ensure that every integer \(0\le j\le4\mu\) lies in the support of \(X\) and \(Y\). On this range, the exact probability ratios are
\[
\frac{\Pr(X=j)}{\Pr(Y=j)}
=\frac{(m)_j}{m^j}
\frac{(N-m)_{F-j}}{(N-m)^{F-j}}
\left(\frac{(N)_F}{N^F}\right)^{-1},
\]
\begin{equation}
\frac{\Pr(Y=j)}{\Pr(Z=j)}
=\frac{(F)_j}{F^j}(1-\mu/F)^{F-j}e^\mu,
\label{eq:L7}
\end{equation}
where \((u)_v=u(u-1)\cdots(u-v+1)\). The inequality
\begin{equation}
\log\frac{(u)_v}{u^v}
\ge-\frac{v(v-1)}{2(u-v+1)}
\label{eq:L8}
\end{equation}
follows by summing \(\log(1-x)\ge-x/(1-x)\). In the first ratio, discard the last factor, which is at least one. By \labelcref{eq:L5}, the magnitudes of the two resulting negative logarithmic bounds are at most \(2/31\) and \(2/15\). Thus \(\Pr(X=j)\ge e^{-1}\Pr(Y=j)\).

For the second ratio use \labelcref{eq:L8} and
\[
\log(1-u)\ge-u-\frac{u^2}{2(1-u)}.
\]
Its logarithm is at least
\[
-\frac{j^2}{2(F-j+1)}
-\frac{\mu^2}{2F(1-\mu/F)}>-1,
\]
since the two negative magnitudes are at most \(2/7\) and \(1/62\). For example, \(F-j+1\ge32\mu^2-4\mu+1\ge28\mu^2\), so \(j^2/[2(F-j+1)]\le16\mu^2/(56\mu^2)=2/7\). Hence
\begin{equation}
\Pr(X=j)\ge e^{-2}\Pr(Z=j)
\qquad(0\le j\le4\mu).
\label{eq:L9}
\end{equation}
Here \(k\le3\mu\). Stirling's inequality \(k!\le3\sqrt{k}(k/e)^k\) and the derivative \(\ln(x/\mu)\) of the Poisson rate \(x\ln(x/\mu)-x+\mu\) imply
\[
\Pr(Z=k)\ge
\frac{1}{9\sqrt{3\mu}}e^{-\mu I_+(\tau)}.
\]
Indeed rounding \((1+\tau)\mu\) up by at most one increases this rate by at most \(\ln3\). Equation \labelcref{eq:L9} proves \labelcref{eq:L6}, for example with \(c=e^{-2}/(9\sqrt3)\). \end{proof}

\begin{proof}[Proof of \cref{thm:lower-bound}]
Put \(m=q^h\), \(B=q^b\), and \(\mu=q^\ell\). Choose \(F=q^r\) to be the smallest power of \(q\) at least \(32\mu^2\), and set \(n=b+r\). Then \(\mu=\Theta_{q,\tau,\eta}(h)\),
\[
r=2\ell+O_q(1),\qquad n=h+\ell+O_q(1),
\]
and \labelcref{eq:L5} holds for sufficiently large \(h\). Furthermore,
\[
\mu I_+(\tau)\le(1-\eta)h\ln q
\le(1-\eta/2)b\ln q
\]
eventually, because \(h-b=O(\log h)\). \cref{lem:reverse-tail} therefore gives
\begin{equation}
Bp_+\ge c\mu^{-1/2}B^{\eta/2}.
\label{eq:L12}
\end{equation}
The elementary Gaussian-binomial bound
\[
{n\brack b}_q\le4q^{b(n-b)}
\]
follows from the product formula and \(\prod_{j\ge1}(1-2^{-j})>1/4\). Thus \(\ln{n\brack b}_q=O_{q,\tau,\eta}(h\log h)\), while the right side of \labelcref{eq:L12} grows exponentially in \(h\). Condition \labelcref{eq:L2} holds, so some source defeats every surjection. Nonsurjective maps have an empty fiber and hence are unbalanced because \(\tau<1\).

Finally, if a map to \(b'>b\) symbols were balanced, composing it with a surjective linear map to \(b\) symbols would remain balanced: each new fiber is the disjoint union of \(q^{b'-b}\) old fibers, and both the target mean and the error bound scale by that factor. This is impossible. Formula \labelcref{eq:L10} gives \labelcref{eq:L11}. \end{proof}

\section{Proof of the Kakeya obstruction}
\label{sec:kakeya-obstruction}

A rank-\(r\) Kakeya set \(K\subseteq\mathbb F_p^d\) contains a translate of every \(r\)-dimensional linear subspace. We use the size bound in \cite[Theorem 4]{KLSS} directly. Its complement yields the following hashing obstruction.

\begin{proposition}[a rank-four Kakeya obstruction]
\label{prop:rank-four-obstruction}
Let \(p\ge5\) be prime and \(d\ge5\). There is a nonempty \(A\subseteq\mathbb F_p^d\) such that every linear map to \(\mathbb F_p^{d-4}\) has an empty fiber and
\begin{equation}
\frac{|A|}{p^{d-4}}
=p^{4-d}\left\lfloor
p^d\left[
1-\left(1-\frac{p-3}{2p^4}\right)^{\lfloor d/5\rfloor}
\right]\right\rfloor.
\label{eq:L14}
\end{equation}
If \(d\to\infty\), \(p\to\infty\), and \(d=o(p^3)\), then this mean is
\begin{equation}
(1+o(1))\frac{dp}{10}.
\label{eq:L15}
\end{equation}
\end{proposition}

\begin{proof}
Theorem 4 of \cite{KLSS}, with ambient dimension \(d\), rank \(4\), and their \(\delta_p=3\), supplies a rank-four Kakeya set \(K\) satisfying
\begin{equation}
|K|\le
p^d\left(1-\frac{p-3}{2p^4}\right)^{\lfloor d/5\rfloor}.
\label{eq:L16}
\end{equation}
Choose \(A\subseteq\mathbb F_p^d\setminus K\) with cardinality equal to the floor in \labelcref{eq:L14}. Every surjection to \(\mathbb F_p^{d-4}\) has a fiber contained in \(K\), so that fiber misses \(A\). Nonsurjections also have an empty fiber. Finally, put \(t=\lfloor d/5\rfloor\) and \(u=(p-3)/(2p^4)\). Under the asymptotic hypotheses, \(tu=o(1)\), and \(1-(1-u)^t=(1+o(1))tu\). The floor changes the mean by less than \(p^{4-d}\), proving \labelcref{eq:L15}. \end{proof}

\begin{proposition}[prescribing the kernel dimension]
\label{prop:kernel-dimension}
Let \(p\ge5\) be prime, \(n\ge r+1\), \(4\le r<n\), put \(d=n-r+4\), and suppose \(d\ge5\). There is \(S\subseteq\mathbb F_p^n\) such that every linear map to \(\mathbb F_p^{n-r}\) has an empty fiber and its mean is exactly the right side of \labelcref{eq:L14}, with this value of \(d\). For every positive integer \(m\le |S|\), we can choose an \(m\)-element subset of \(S\) having an empty fiber under every such map.
\end{proposition}

\begin{proof}
Construct \(A\subseteq\mathbb F_p^d\) by \cref{prop:rank-four-obstruction} and set
\begin{equation}
S=A\times\{0\}^{r-4}\subseteq\mathbb F_p^n.
\label{eq:L17}
\end{equation}
Restrict any map \(\mathbb F_p^n\to\mathbb F_p^{n-r}=\mathbb F_p^{d-4}\) to the first \(d\) coordinates. If the restriction is surjective, \cref{prop:rank-four-obstruction} supplies an empty output fiber on \(A\); if it is not, an output outside its image is empty. The cardinality and output dimension have not changed, so the mean is preserved. Taking an arbitrary subset preserves all empty fibers and permits every smaller positive integer cardinality. \end{proof}

\begin{proof}[Proof of \cref{cor:quadratic-obstruction}]
Along \labelcref{eq:L18}, we have \(r\ge4\), \(r=o(n)\), and \(d=n-r+4\sim n\) for all sufficiently large parameters. Apply \cref{prop:kernel-dimension}. Since \(d=o(p^3)\), the asymptotic estimate \labelcref{eq:L15} gives
\[
\mu=(1+o(1))np/10.
\]
Moreover \(r\ln p=cn+O(\ln p)\) and \(\ln\mu=O(\ln p)\), so
\[
|S|p^{-n}=\mu p^{-r}=\exp(-cn+O(\ln p)).
\]
\cref{prop:kernel-dimension} gives every smaller positive cardinality by taking subsets. An empty fiber violates balance for every \(\tau<1\), regardless of the allowed failure probability \(\delta<1\). Thus a threshold \(o(n^2)\) cannot suffice. \end{proof}

\section{Proof of the smooth-covering theorem}
\label{sec:covering-proof}

We follow the construction of Ordentlich, Regev, and Weiss \cite[Theorem 1.2]{ORW}. They begin with a Haar--Siegel random lattice whose packing and covering dilation factors for the body are comparable. After rescaling, they obtain a coarse lattice that packs a slightly expanded copy of the body and a finer lattice that serves as a grid. They discretize this expanded copy and a slightly contracted copy on the grid. The packing condition makes reduction modulo the coarse lattice injective on both copies, so their grid points become two subsets of a finite vector space. A random subspace determines an intermediate lattice, and its coset intersections with these two sets count lattice points in grid translates. Two-sided hashing controls these counts; a convexity argument then controls every real translate of the original body. Finally, the normalized intermediate lattice again has the Haar--Siegel distribution, which turns this construction into a statement about a random unimodular lattice.

Our main change is to replace the hashing input in the proof of \cite[Theorem 3.4]{ORW} by \cref{cor:linear-field-cutoff}. At fixed discrepancy and failure probability, this reduces the sufficient mean fiber size from order \(p^3\) to order \(p^2\). We choose \(p\asymp_\varepsilon n\) and arrange that the mean fiber sizes are comparable to \(\operatorname{vol}(K)\), giving the quadratic volume threshold. We also use a fixed cutoff for the ratio of covering to packing radii, absorb its exponentially small failure probability into the dimension cutoff, and handle very large volumes by direct lattice-point counting. We use the following forms of their geometric and probabilistic results.

We use the notation \(N(\Lambda,K,x)\), \(\eta(K,\Lambda)\), and \(\mu_n\) from \cref{sec:smooth-coverings}. As in \cite{ORW}, let \(r_{\mathrm{pack},K}(L)\) and \(r_{\mathrm{cov},K}(L)\) be the supremum of packing dilation factors and the infimum of covering dilation factors, respectively, for \(K\) with lattice \(L\), and write
\begin{equation}
\rho_K(L)=\frac{r_{\mathrm{cov},K}(L)}{r_{\mathrm{pack},K}(L)}.
\label{eq:C5}
\end{equation}

\textbf{Radius-ratio estimate (\cite[Proposition 3.6]{ORW}).} For every compact convex body \(K\subseteq\mathbb R^n\) with nonempty interior and every \(\alpha>0\),
\[
\Pr_{L\sim\mu_n}\bigl(\rho_K(L)\ge2\alpha^2\bigr)
<3\left(\frac2\alpha\right)^n.
\]
In particular, \(\alpha=4\) gives a fixed radius-ratio cutoff of \(32\) with failure probability less than \(3\cdot2^{-n}\).

\textbf{Discretization (\cite[Proposition 3.1 and Lemma 3.2]{ORW}).} Let \(G\subseteq\mathbb R^n\) be a lattice and let \(0<s<1\). If \(G+sK=\mathbb R^n\), then every \(x\in\mathbb R^n\) admits \(y_-,y_+\in G\) such that
\[
(1-s)K+y_-\subseteq K+x\subseteq(1+s)K+y_+.
\]
Consequently, for any lattice \(L\), lower bounds on \(N(L,(1-s)K,y)\) and upper bounds on \(N(L,(1+s)K,y)\) that hold for every \(y\in G\) also bound \(N(L,K,x)\) for every \(x\in\mathbb R^n\). We use the proposition in this lattice case; its source statement allows any discrete set \(G=-G\).

\textbf{Lattice-point counting (\cite[Lemma 3.3]{ORW}).} Let \(G\subseteq\mathbb R^n\) be a lattice and \(D\subseteq\mathbb R^n\) a compact convex body with nonempty interior. If \(G+\beta D=\mathbb R^n\) for some \(0<\beta<1\), then
\[
(1-\beta)^n\frac{\operatorname{vol}(D)}{\det G}
\le |G\cap D|
\le(1+\beta)^n\frac{\operatorname{vol}(D)}{\det G}.
\]
Neither this estimate nor the discretization statement requires symmetry or a choice of origin inside the body.

\textbf{Random-neighbor invariance (\cite[Proposition 2.1]{ORW22}).} Fix a prime \(p\) and an integer \(1\le r\le n\). Sample \(L\sim\mu_n\), and then choose \(L'\) uniformly among the lattices satisfying \(L\subseteq L'\subseteq p^{-1}L\) and \([L':L]=p^r\). The normalized lattice \(p^{r/n}L'\) has distribution \(\mu_n\). Equivalently, \(L'/L\) is a uniform \(r\)-dimensional subspace of \(p^{-1}L/L\cong\mathbb F_p^n\). This is the invariance statement used in the proof of \cite[Theorem 1.2]{ORW}.

\begin{proof}[Proof of \cref{thm:smooth-covering}]
We prove the theorem with \(C=2^{52}\). Write \(V=\operatorname{vol}(K)\) and \(R=32\), and choose \(n_0\ge5\) as in the theorem. Since \(2^n/n\) increases for integers \(n\ge2\), for all \(n\ge n_0\) we have
\begin{equation}
\left(2^{19}n/\varepsilon\right)^{1/n}\le2,
\qquad 3\cdot2^{-n}<\Delta/2.
\label{eq:C11}
\end{equation}
By Bertrand's postulate choose a prime \(p\) satisfying
\begin{equation}
2^{18}n/\varepsilon\le p<2^{19}n/\varepsilon,
\qquad \tau=\varepsilon/2,
\qquad \delta=\Delta/4.
\label{eq:C12}
\end{equation}
Our volume hypothesis implies
\begin{equation}
V\ge\frac{6144p^2}{\varepsilon^2\Delta},
\label{eq:C25}
\end{equation}
since \(6144\cdot2^{38}=3\cdot2^{49}<2^{52}\). Define
\begin{equation}
r_0=\left\lceil\log_pV+n\log_p(4R)\right\rceil,
\quad r=\min\{r_0,n\},
\quad c=4R^2p^{1/n},
\quad s=c/p.
\label{eq:C13}
\end{equation}
These choices are deterministic, independent of the lattice to be sampled, and give
\begin{equation}
c\le8192,\qquad s\le\frac{\varepsilon}{32n}.
\label{eq:C15}
\end{equation}

Sample \(L_0\sim\mu_n\), put \(\Gamma=p^{r/n}L_0\) and \(L_f=p^{-1}\Gamma\), and let \(\pi_\Gamma:L_f\to L_f/\Gamma\cong\mathbb F_p^n\) be the quotient map. For every \(L_0\), sample a uniform \(r\)-dimensional subspace \(U\) of this quotient using independent randomness, and set \(\Lambda_U=\pi_\Gamma^{-1}(U)\). Thus \(\det\Lambda_U=1\), and \(\Lambda_U=L_0\) when \(r=n\). Proposition 3.6 of \cite{ORW}, with \(\alpha=4\), gives
\begin{equation}
\Pr\bigl(\rho_K(L_0)\ge R\bigr)<3\cdot2^{-n}<\Delta/2.
\label{eq:C9}
\end{equation}
Work first on the event \(\rho_K(L_0)<R\). The packing and covering volume inequalities imply
\begin{equation}
r_{\mathrm{pack},K}(L_0)\ge V^{-1/n}/R,
\qquad r_{\mathrm{cov},K}(L_0)\le RV^{-1/n}.
\label{eq:C16}
\end{equation}

Suppose \(r_0<n\), so \(r=r_0\), and put \(b=n-r_0\). Equations \labelcref{eq:C13} and \labelcref{eq:C16} give
\begin{equation}
\det\Gamma=p^{r_0},\qquad
r_{\mathrm{pack},K}(\Gamma)\ge4,\qquad
r_{\mathrm{cov},K}(\Gamma)<c.
\label{eq:C18}
\end{equation}
In particular, \(L_f+sK=\mathbb R^n\), and reduction modulo \(\Gamma\) is injective on each of \((1-s)K\) and \((1+s)K\).

Use the two discretized sources from the proof of \cite[Theorem 3.4]{ORW}:
\begin{equation}
A_\pm=\pi_\Gamma\bigl(L_f\cap((1\pm s)K)\bigr),
\qquad \mu_\pm=|A_\pm|/p^b.
\label{eq:C19}
\end{equation}
Since \(\det L_f=p^{-b}\) and \(L_f+sK=\mathbb R^n\), apply \cite[Lemma 3.3]{ORW} to \((1\pm s)K\) with \(\beta_\pm=s/(1\pm s)\). This recovers their equation (25):
\begin{equation}
(1-2s)^nV\le\mu_\pm\le(1+2s)^nV.
\label{eq:C20}
\end{equation}
Equations (28)--(32) in the same proof show that, whenever \(U\) is \(\tau\)-shift-balanced for both sources,
\begin{equation}
\eta(K,\Lambda_U)<\tau+8ns\le3\varepsilon/4<\varepsilon.
\label{eq:C23}
\end{equation}
These parts of their argument use only the packing, covering, and balance assumptions. We replace their subsequent application of the finite-field theorem as follows.

By \labelcref{eq:C15}, \((1-2s)^n\ge1-\varepsilon/16>1/2\), so \labelcref{eq:C20} and \labelcref{eq:C25} imply
\begin{equation}
\mu_\pm>V/2\ge\frac{3072p^2}{\varepsilon^2\Delta}
=\frac{192p^2}{\tau^2\delta}.
\label{eq:C26}
\end{equation}
Also \(V>p^2\), hence \(r_0\ge3\), and therefore \(b\ge1\) and \(b+2<n\). The field choice gives \(p\ge16(b+2)\). \cref{cor:linear-field-cutoff} now applies to both sources. A uniform subspace has the same distribution as the kernel of a uniform surjection, so the probability that \(U\) fails to be \(\tau\)-shift-balanced for either source is less than \(2\delta=\Delta/2\), conditional on each \(L_0\) with \(\rho_K(L_0)<R\).

If \(r_0\ge n\), then \(r=n\) and \(\Lambda_U=L_0\). The definition of \(r_0\) gives \(V>p^{n-1}/(4R)^n\), whence \labelcref{eq:C16} implies
\begin{equation}
r_{\mathrm{cov},K}(L_0)<4R^2p^{-1+1/n}=s.
\label{eq:C28}
\end{equation}
Apply \cite[Lemma 3.3]{ORW} to each translate of \(K\), with \(\beta=s\), to obtain
\[
(1-s)^n\le\frac{N(L_0,K,x)}V\le(1+s)^n
\qquad(x\in\mathbb R^n).
\]
Since \(s\le\varepsilon/(32n)\), these bounds give \(\eta(K,L_0)<\varepsilon\).

Finally, \(\Lambda_U\) in the first branch is a normalized random \(p\)-neighbor of \(L_0\), and the second branch is the same construction with \(r=n\). Since \(p,r\) were fixed before sampling \(L_0\), \cite[Proposition 2.1]{ORW22} implies that the resulting lattice has law \(\mu_n\), exactly as in the proof of \cite[Theorem 1.2]{ORW}. Combining \labelcref{eq:C9} with the conditional failure bound gives total failure probability less than
\begin{equation}
3\cdot2^{-n}+\Delta/2<\Delta.
\label{eq:C29}
\end{equation}
This proves \cref{thm:smooth-covering}. \end{proof}

\begin{corollary}[existence at quadratic density]
\label{cor:existence-density}
For each fixed \(0<\varepsilon<1\), every sufficiently high-dimensional convex body admits an \(\varepsilon\)-smooth lattice covering of density \(O(\varepsilon^{-4}n^2)\), with an absolute implicit constant and a dimension cutoff depending on \(\varepsilon\).
\end{corollary}

\begin{proof}
Apply \cref{thm:smooth-covering} with \(\Delta=1/2\) to a dilate of the given body having volume \(2C\varepsilon^{-4}n^2\), and choose a unimodular lattice for which the conclusion holds. Rescaling the lattice back gives the claim. \end{proof}

\section*{LLM Use}

The author used OpenAI's GPT-6 Astra to assist with literature review, drafting and revising the exposition, and developing and checking proofs. The author takes full responsibility for the content of this paper, including the correctness of its mathematical arguments and the accuracy of its references.

\bibliographystyle{alpha}
\bibliography{two-sided-hashing-smooth-coverings}

@article{ADMPT,
  author = {Alon, Noga and Dietzfelbinger, Martin and Miltersen, Peter Bro and Petrank, Erez and Tardos, G{\'a}bor},
  title = {Linear Hash Functions},
  journal = {Journal of the ACM},
  volume = {46},
  number = {5},
  pages = {667--683},
  year = {1999},
  note = {\url{https://www.cs.tau.ac.il/~nogaa/PDFS/linhash13.pdf}}
}

@article{DD,
  author = {Dhar, Manik and Dvir, Zeev},
  title = {Linear Hashing with {$\ell_\infty$} guarantees and two-sided {Kakeya} bounds},
  journal = {TheoretiCS},
  volume = {3},
  number = {8},
  pages = {1--30},
  year = {2024},
  note = {We use the published numbering. \url{https://theoretics.episciences.org/13338/pdf}}
}

@article{DDL,
  author = {Dhar, Manik and Dvir, Zeev and Lund, Ben},
  title = {Simple Proofs for {Furstenberg} Sets Over Finite Fields},
  journal = {Discrete Analysis},
  pages = {1--16},
  year = {2021},
  note = {Article 22. \url{https://arxiv.org/abs/1909.03180v2}}
}

@misc{DLMNRR,
  author = {Doron, Dean and Leonov, Tal and Mosheiff, Jonathan and Navas, Henrique and Resch, Nicolas and Ribeiro, Jo{\~a}o},
  title = {Discrepancy for Random Linear Codes},
  howpublished = {arXiv:2606.24471v1},
  year = {2026},
  note = {\url{https://arxiv.org/abs/2606.24471v1}}
}

@article{Dvi,
  author = {Dvir, Zeev},
  title = {On the size of {Kakeya} sets in finite fields},
  journal = {Journal of the American Mathematical Society},
  volume = {22},
  number = {4},
  pages = {1093--1097},
  year = {2009},
  note = {\url{https://arxiv.org/abs/0803.2336}}
}

@article{EE,
  author = {Ellenberg, Jordan S. and Erman, Daniel},
  title = {{Furstenberg} sets and {Furstenberg} schemes over finite fields},
  journal = {Algebra \& Number Theory},
  volume = {10},
  number = {7},
  pages = {1415--1436},
  year = {2016},
  note = {\url{https://arxiv.org/abs/1502.03736}}
}

@article{EOT,
  author = {Ellenberg, Jordan S. and Oberlin, Richard and Tao, Terence},
  title = {The {Kakeya} set and maximal conjectures for algebraic varieties over finite fields},
  journal = {Mathematika},
  volume = {56},
  number = {1},
  pages = {1--25},
  year = {2010},
  note = {\url{https://arxiv.org/abs/0903.1879v2}}
}

@inproceedings{JKZ,
  author = {Jaber, Michael and Kumar, Vinayak M. and Zuckerman, David},
  title = {Linear Hashing Is Optimal},
  booktitle = {Proceedings of {STOC} 2025},
  pages = {245--255},
  year = {2025},
  note = {\url{https://arxiv.org/abs/2505.14061v1}}
}

@article{KLSS,
  author = {Kopparty, Swastik and Lev, Vsevolod F. and Saraf, Shubhangi and Sudan, Madhu},
  title = {{Kakeya}-type sets in finite vector spaces},
  journal = {Journal of Algebraic Combinatorics},
  volume = {34},
  pages = {337--355},
  year = {2011},
  note = {\url{https://www.math.utoronto.ca/swastik/kakeya2.pdf}}
}

@techreport{KM,
  author = {Kumar, Vinayak M. and Mon, Geoffrey},
  title = {List Decoding, Linear Hashing, and {Furstenberg} over {$\mathbb F_q$}},
  institution = {Electronic Colloquium on Computational Complexity},
  type = {Report},
  number = {TR26-181},
  year = {2026},
  note = {15 September 2026. arXiv:2609.17020v1. \url{https://eccc.weizmann.ac.il/report/2026/181/}}
}

@article{ORW,
  author = {Ordentlich, Or and Regev, Oded and Weiss, Barak},
  title = {Bounds on the density of smooth lattice coverings},
  journal = {Journal d'Analyse Math{\'e}matique},
  volume = {156},
  pages = {301--326},
  year = {2025},
  note = {References to numbered statements use arXiv:2311.04644v1 (November 2023). \url{https://arxiv.org/abs/2311.04644v1}}
}

@article{ORW22,
  author = {Ordentlich, Or and Regev, Oded and Weiss, Barak},
  title = {New bounds on the density of lattice coverings},
  journal = {Journal of the American Mathematical Society},
  volume = {35},
  pages = {295--308},
  year = {2022},
  note = {\url{https://arxiv.org/abs/2006.00340}}
}

@misc{PB,
  author = {Pathegama, Madhura and Barg, Alexander},
  title = {{R{\'e}nyi} divergence guarantees for hashing with linear codes},
  howpublished = {arXiv:2405.04406v2},
  year = {2025},
  note = {Version 2, 5 June 2025; first version May 2024. \url{https://arxiv.org/abs/2405.04406v2}}
}
\end{document}